\date{}

\date{}

\title{Parabolic-hyperbolic dichotomy through half-plane coexistence}
\author{\'Ad\'am Tim\'ar }

\documentclass[11pt,naturalnames]{article}
\renewcommand\footnotemark{}
\usepackage[ruled,vlined]{algorithm2e}
\usepackage{amsmath}
\usepackage{amsthm}
\usepackage{amsfonts}
\usepackage{graphicx}
\usepackage[percent]{overpic}
\usepackage{xcolor}
\newif\ifhyper\IfFileExists{hyperref.sty}{\hypertrue}{\hyperfalse}
\ifhyper\usepackage{hyperref}\fi
\usepackage{enumitem}
\usepackage{subfig}
\usepackage{caption}
\usepackage{keyval}
\usepackage[margin=1in]{geometry}
\usepackage{setspace}

\usepackage[displaymath, mathlines,pagewise]{lineno}
\definecolor{myblue}{RGB}{10,9,224}
\definecolor{mygreen}{RGB}{38,175,79}

\theoremstyle{definition}

\newtheorem{theorem}{Theorem}

\newtheorem{lemma}[theorem]{Lemma}

\newtheorem{remark}[theorem]{Remark}

\def \proof {{ \medbreak \noindent {\bf Proof.} }}
\def\proofof#1{{ \medbreak \noindent {\bf Proof of #1.} }}
\def\proofcont#1{{ \medbreak \noindent {\bf Proof of #1, continued.} }}
\def\supp{{\rm supp}}
\def\max{{\rm max}}
\def\min{{\rm min}}
\def\dist{{\rm dist}}
\def\Aut{{\rm Aut}}
\def\id{{\rm id}}
\def\Stab{{\rm Stab}}
\def\m{{m}}

\def\calH{{\mathcal H}}

\begin{document}
\maketitle

\bigskip

\def\eref#1{(\ref{#1})}
\newcommand{\Prob} {{\bf P}}
\newcommand{\C}{\mathcal{C}}
\newcommand{\LL}{\mathcal{L}}
\newcommand{\Z}{\mathbb{Z}}
\newcommand{\N}{\mathbb{N}}
\newcommand{\HH}{\mathbb{H}}
\newcommand{\Rr}{\mathbb{R}^3}
\newcommand{\h}{\mathcal{H}}
\newcommand{\calS}{\mathbb{S}}
\def\diam{\mathrm{diam}}
\def\length{\mathrm{length}}
\def\ev#1{\mathcal{#1}}
\def\Isom{{\rm Isom}}
\def\Re{{\rm Re}}
\def \eps {\epsilon}
\def \P {{\Bbb P}}
\def \E {{\Bbb E}}
\def \proof {{ \medbreak \noindent {\bf Proof.} }}
\def\proofof#1{{ \medbreak \noindent {\bf Proof of #1.} }}
\def\proofcont#1{{ \medbreak \noindent {\bf Proof of #1, continued.} }}
\def\supp{{\rm supp}}
\def\max{{\rm max}}
\def\min{{\rm min}}
\def\dist{{\rm dist}}
\def\Aut{{\rm Aut}}
\def\id{{\rm id}}
\def\Stab{{\rm Stab}}

\newcommand{\lra}{\leftrightarrow}
\newcommand{\xlra}{\xleftrightarrow}
\newcommand{\xnlra}{\xnleftrightarrow}
\newcommand{\pc}{{p_c}}
\newcommand{\pt}{{p_T}}
\newcommand{\ptk}{{\hat{p}_T}}
\newcommand{\pl}{{\tilde{p}_c}}
\newcommand{\pe}{{\hat{p}_c}}
\newcommand{\pr}{\mathrm{\mathbb{P}}}
\newcommand{\pp}{\mu}
\newcommand{\ex}{\mathrm{\mathbb{E}}}
\newcommand{\ee}{\mathrm{\overline{\mathbb{E}}}}

\newcommand{\om}{{\omega}}
\newcommand{\ebd}{\partial_E}
\newcommand{\ivbd}{\partial_V^\mathrm{in}}
\newcommand{\ovbd}{\partial_V^\mathrm{out}}
\newcommand{\q}{q}
\newcommand{\TT}{\mathfrak{T}}
\newcommand{\T}{\mathcal{T}}
\newcommand{\RR}{\mathcal{R}}

\newcommand{\CC}{\Pi}
\newcommand{\BB}{\Pi}

\newcommand{\ELL}{\mathcal{L}}
\newcommand{\A}{\mathcal{A}}
\newcommand{\cc}{\mathbf{c}}
\newcommand{\pa}{{PA}}
\newcommand{\degi}{\deg^{in}}
\newcommand{\dego}{\deg^{out}}
\def\Pn{{\bf P}_n}

\newcommand{\R}{\mathbb R}

\newcommand{\F}{F}
\newcommand{\FF}{\mathfrak{F}}
\newcommand{\Can}{\rm Can}
\newcommand{\Vol}{\mathrm Vol}
\def\br{{\rm br}}
\def\gr{{\rm gr}}
\def\Gstar{{{\cal G}_{*}}}
\def\Gstarstar{{{\cal G}_{**}}}
\def\Gstarplus{{{\cal G}_{*}^\frown}}
\def\Rel{{\cal R}}
\def\Comp{{\rm Comp}}
\def\calG{{\cal G}}
\def\calB{{\cal B}}
\def\omps{{\omega_\delta^\eps}}
\def\HP{{H\! P}}
\def\ohp{{\HP_\ell(\omega)}}
\def\ohpdual{{\HP_\ell(\omega^\star)}}

\def\ohpinf{{\HP_\ell(\omega)^\infty}}
\def\ohpdualinf{\HP_\ell(\omega^\star)^\infty}


\begin{abstract}
Consider a unimodular random planar map (URM) with an invariant ergodic percolation having infinite primal and dual clusters. We say that there is half-plane coexistence if both the percolation and its dual have infinite clusters when restricted to a half-plane. Under mild assumptions on the percolation, we show that the URM is parabolic if and only if there is no half-plane coexistence, and it is hyperbolic if and only if there is half-plane coexistence. This extends the recent half-plane non-coexistence result for $\mathbb{Z}^2$ by Klausen and Kravitz and provides another manifestation of the parabolic-hyperbolic dichotomy for URM's.
\end{abstract}

\bigskip

A unimodular random map (URM) is a unimodular random graph (URG) that is almost surely planar and locally finite, equipped with a proper embedding into either 
the Euclidean plane $\R^2$ or the hyperbolic plane $\HH^2$ with bounded faces and no accumulation points, whose distribution is invariant under the isometries of the ambient plane. Examples of URM's are 1-ended transitive and quasi-transitive planar graphs, Voronoi tessallations of invariant point processes, or Benjamini-Schramm limits of suitable finite planar graphs, such as the Uniform Infinite Planar Triangulation (UIPT) and Quadrangulation (UIPQ), and their hyperbolic variants; see \cite{C}. Whether the Euclidean or the hyperbolic plane is under examination, that is, whether the map is {\it parabolic} or {\it hyperbolic}, can be identified through a variety of other properties. Such characterization was started by He and Schramm in the context of circle packings of infinite triangulations of the plane \cite{HS,HS2}, continued by Benjamini and Schramm \cite{BS2, BS3}, and systematically extended by Angel, Hutchcroft, Nachmias and Ray \cite{AHNR, AHNR2}. One example of the sharp dichotomy is according to invariant amenability, as defined in \cite{AL}: invariantly amenable URM's are necessarily maps embedded in the Euclidean plane, whereas invariantly nonamenable URM's are maps embedded in the hyperbolic plane \cite{AHNR2}. The present paper contributes to the description of this dichotomy, by adding one more property that separates parabolic and hyperbolic URM's.

Given a random subgraph $\omega$, say that {\it there is coexistence} if the restriction of $\omega$ and its dual to the upper half-plane both contain some infinite clusters.
We will consider two conditions:

(A) $\omega$ has finitely many infinite clusters almost surely;

(B) the percolation is insertion and deletion tolerant.

\begin{theorem}\label{main}
\noindent
Suppose $G$ is a URM with dual. Let $\omega$ be a random subgraph of $G$ given by an ergodic invariant edge percolation that satisfies (A) or (B). Suppose further that both $\omega$ and $\omega^\star$ have some infinite cluster. Then

(1) $G$ is parabolic if and only if there is no half-plane coexistence almost surely;

(2) $G$ is hyperbolic if and only if there is half-plane coexistence almost surely.
\end{theorem}

All definitions will be given in the next section. 

Among the other characterizing properties for the parabolic-hyperbolic dichotomy are whether the expected curvature is zero, the critical percolation parameters $p_c$ and $p_u$ are equal a.s., or whether every subtree is recurrent a.s.; see the Dichotomy Theorem in \cite{AHNR2} for a complete list and for the history of results. 
Several properties in the Dichotomy Theorem of \cite{AHNR2} apply to URG's without any reference to planar embeddings, invariant amenability being one of them. The representability of such planar URG's as URM's in $\R^2$ or $\HH^2$
was addressed in \cite{T,BT}. 

\begin{remark}\label{megj}
The proof of theorem \ref{main}  works almost automatically if we replace the full group of isometries by any other transitive group of isometries. It is only the definition of $\phi$ in the proof of Theorem \ref{main} that has to be suitably modified.
\end{remark}

Klausen and Kravitz \cite{KK} proposed to study half-plane non-coexistence on the lattice $\Z^2$ for general percolation processes, where FKG-type correlation inequalities, which were used in earlier non-coexistence theorem, are not available. Their result inspired the present work, and follows from our Theorem \ref{main} by Remark \ref{megj}. 
\begin{theorem}\label{old}[Klausen-Kravitz, \cite{KK}]   
Consider a random subgraph $\omega$ of $\Z^2$ given by a translation invariant ergodic edge percolation. Suppose that at least one of the following holds:

\noindent
(A) $\omega$ has finitely many infinite clusters almost surely;

\noindent
(B) the percolation is insertion and deletion tolerant.

\noindent
Then the restriction of $\omega$ or its dual to the upper half-plane has no infinite cluster. 
\end{theorem}
We mention that in \cite{KK} the slightly stronger condition of uniform finite energy replaced (B) above.

Some condition on the percolation in Theorem \ref{main} is necessary, as shown by the example of all horizontal lines in $\Z^2$. (This is only translation invariant, but simple isometry-invariant examples also exist.) However, the claim about the hyperbolic plane works in full generality.

\begin{theorem}\label{hyper}
Suppose $G$ is a URM with dual in $\HH^2$. Let $\omega$ be a random subgraph of $G$ given by an ergodic invariant edge percolation and suppose that both $\omega$ and $\omega^\star$ have some infinite cluster. Then
there is half-plane coexistence almost surely.
\end{theorem}

\section{Definitions and preliminaries}

Example 9.6 in \cite{AL} explains why duals of such graphs are always URG's.
We say that $G$ is a URM {\it with dual} when $G$ and its dual $G^\star$ are embedded in the plane $\R^2$ or $\HH^2$ in a jointly isometry-invariant way, with no accumulation points, every vertex of $G^\star$ is in a face of $G$ and vica versa, and an edge of $G^\star$ and an edge of $G$ intersect only if they are dual. With a slight abuse of terminology, we will refer to a vertex or edge both as an object in the graph and as a subset of the plane.
When we refer to such a $G$ we will always implicitly assume that $G^\star$ is also given. Given some random subgraph $\omega$ of $G$, its dual $\omega^\star\subset G^\star$ is automatically defined as a subgraph of $V(G^\star)$ where we delete those edges whose dual is in $\omega$. An {\it invariant edge percolation} on $G$ is a random spanning subgraph of $G$ such that its joint distribution with $G$ is invariant under the isometries of the plane. Similarly we can define a random subgraph of a percolation configuration to be invariant if it is jointly invariant with the percolation and $G$.
If $G=(V,E)$ is a random graph of distribution $\mu_G$ and $\omega \in \{0,1\}^E$ an edge-percolation configuration, where $\omega(e)=1$ indicates that $e$ is open, a probability measure $\mathbb{P}$ on $\{0,1\}^E$ is called \emph{insertion tolerant} $\mu_G$-almost surely for every edge $e \in E(G)$ and every event $A$ with $\mathbb{P}(A)>0$,
we have
$\mathbb{P}(A \cap \{\omega(e)=1\}) > 0$.
Similarly, $\mathbb{P}$ is called \emph{deletion tolerant} if under the above conditions we have $\mathbb{P}(A)>0$, $\mathbb{P}(A \cap \{\omega(e)=0\}) > 0$.

Given a half-plane $\HP$ in $\R^2$ or $\HH^2$, and an embedded graph $\gamma$, the {\it restriction of $\gamma$ to $\HP$} is the subgraph of $\gamma$ given by vertices and edges that are fully contained in $\HP$. Denote this restriction by $\HP(\gamma)$.

For a subgraph $\omega$ of the square lattice $\Z^2$, and a horizontal line $\ell$ in $\Z^2$ or in $\Z^{2\star}$, let $\HP_\ell$ be the closed upper half plane detemined by $\ell$. In accordance with the definition in the previous paragraph, $\HP_\ell(\Z^2)$ is the subgraph of $\Z^2$ restricted to $\HP_\ell$, and $\HP_\ell(\Z^{2\star})$ is the subgraph of $\Z^{2\star}$ restricted to $\HP_\ell$.
For the unimodular random graph underlying a URM $G$, we will need a consequence of the Mass Transport Principle \cite{BLPS, AL}, which we will refer to as MTP. Namely, that there is no way to define a random invariant function (``assignment'') from a subset of $V(G)$ to $V(G)$ in such a way that with positive probability there is a vertex whose preimage is infinite. 

The two claims of the following lemma are intuitively clear.
\begin{lemma}\label{2ended}
Let $G$ be an infinite planar graph with finite degrees, bounded faces and no accumulation points. Let $C_1,C_2\subseteq G$ be vertex-disjoint subgraphs. Then the dual of $C_1\cup C_2$ contains a bi-infinite path that separates $C_1$ from $C_2$ in the plane.

If a subgraph of $G$ has infinitely many 2-ended infinite clusters, then the dual of $\xi$ has infinitely many infinite clusters.
\end{lemma}

\begin{proof}
By the Jordan curve theorem there is a simple curve $\gamma$ in the plane such that $C_1$ is on one side of the curve and $C_2$ is on the other and $\gamma$ has no accumulation point. Reading the edges in the order they are crossed by $\gamma$ defines a biinfinite walk in $G^\star$. Delete cycles starting from some vertex in both directions separately, to obtain a (simple) biinfinite path, which still separates $C_1$ and $C_2$. 

For the second assertion, let $K_1, K_2,\ldots$ be the components of the given subgraph.
Proving by induction, let $U$ be the unbounded component of the plane that contains $K_{n+1}$ after removing $K_1,\ldots,K_n$. Then $K_{n+1}$ splits $U$ into at least two unbounded components. Thus, removing each additional 2-ended cluster increases the number of unbounded complementary components by at least one.
\end{proof}\qed


\section{Proofs}

Let us start with the proof of Theorem \ref{old}. This is a special, simpler case of the proof of the parabolic part of Theorem
\ref{main} and is not needed for the rest of the paper. We decided to include it as a separate proof because of its simplicity, and because it shows the main ideas without some of the technicalities.

Let $\Z^2$ denote the square lattice as a plane graph with vertex set $\{(a,b):a,b\in\Z\}$ and straight line segments as edges. Let its dual $\Z^{2\star}$ be its translate by $(\frac{1}{2},\frac{1}{2})$, with the natural identification between each face of $\Z^2$ and the vertex of the dual that it contains. Given an edge $e$, let $e^\star$ denote its dual edge. 
A {\it horizontal line} in $\Z^2$ is a subgraph induced by $\{(a,b):a\in\Z\}$ for some $b\in\Z$. A horizontal line in $\Z^{2\star}$ is a subgraph induced by $\{(a+\frac{1}{2},b+\frac{1}{2}):a\in\Z\}$ for some $b\in\Z$.



\begin{proofof}{Theorem \ref{old}}
Denote by $\ohpinf$ (respectively $\ohpdualinf$) the union of the infinite clusters of $\ohp$ (respectively $\ohpdual$). Suppose by contradiction that none of $\ohpinf$ and $\ohpdualinf$ is empty.
Fix a horizontal line $\ell$. 

\smallskip
{\bf Case 1: $\ell\cap \ohpinf$ is finite.} This has zero probability, otherwise one can assign to each vertex of $\ell$ a randomly chosen element of $\ell\cap \ohpinf$ and do so for every such $\ell$, to get a MTP contradiction.

\smallskip
{\bf Case 2: $\ell\cap \ohpinf$ is infinite, and there is an infinite cluster $C$ of $\ohpinf$ with $C\cap\ell$ infinite.} 
By the MTP there cannot be a first or last element of $C\cap\ell$ in $\ell$ (with respect to the linear ordering of $\ell$ coming from $\Z$).
For any two points of $C\cap\ell$ there is a path in $C$ between them. The union of such paths over all pairs separates 
$\ohpdualinf$ from $\Z^{2\star}\setminus \HP_\ell(\Z^{2\star})$ in $\Z^{2\star}$.
In particular, every infinite cluster $\hat C$ of $\ohpdual$ is also a cluster of $\omega^\star$. Thus $\hat C$ {\it has a lowermost line}, meaning that there is a lowermost horizontal line $\m $ in $\Z^{2\star}$ that $\hat C$ intersects. This intersection cannot be finite by the MTP. However, if it is infinite and (B) holds, by changing the status of a suitable edge that connects this lowermost line to the line below it, 
we find that with positive probability there exists a horizontal line in $\Z^{2\star}$ and cluster such that the cluster intersects the horizontal line in infinitely many vertices, but only one edge connects it to a lower line.
Assign the endpoint of this special edge in the intersection of the mentioned cluster and horizontal line (and for each such pair of horizontal line and cluster), to get a MTP contradiction. So when (B) holds, we can rule out Case 2. Assume now (A), and that $\hat C\cap\m $ is infinite.
Then $\hat C$ is the only infinite cluster of $\omega^\star$ that has $\m $ as its lowermost line and intersects it.
If the number of infinite clusters in $\omega^\star$ that have a lowermost line is finite (and at least one), then we identified finitely many horizontal lines of $\Z^2$ in an invariant way, giving a MTP contradiction. (Let every vertex $v$ be mapped to one of the vertices from these lines and on the vertical line containing $v$, chosen at random.) If their number is infinite, then we can apply Lemma \ref{2ended} to conclude that $\omega$ has infinitely many clusters, a contradiction.

\smallskip
{\bf Case 3: $\ell\cap \ohpinf$ is infinite, and for every infinite cluster $C$ of $\ohpinf$, $C\cap\ell$ is finite.} In particular, $\ohpinf$ has infinitely many infinite clusters. This is not possible under the assumption (B), because we could close finitely many edges and create an infinite cluster with finitely many lowermost points, a contradiction to the MTP. So let us assume (A). As $\omega$ has only finitely many clusters, infinitely many infinite clusters of $\ohpinf$ belong to the same $\omega$-cluster $K$. But these can all be separated from $K$ by removing their finite intersection with $\ell$, showing that $K$ has infinitely many ends. This is not possible for an invariant percolation on an amenable graph by a Burton-Keane argument (see also Section 5 in \cite{BLPS}), giving a final contradiction.
\end{proofof}\qed
\medskip

The proof of Theorem \ref{main} in the parabolic setting will follow the same track.

\begin{proofof}{Theorem \ref{main}}
It is enough to prove that for $G$ and $\omega$ as in the theorem, if $G$ is parabolic then there is no half-plane coexistence, and if $G$ is hyperbolic then there is always half-plane coexistence. Let us start with the first one.

{\bf Suppose that $G$ is a parabolic URM in the Euclidean plane.} Consider a copy of $\Z^2$ as a graph embedded in $\R^2$, with the horizontal and vertical axes and their positive directions identified. 
Let $\phi$ be the random isometry, independent from $G$ and $\omega$, which is the composition of a rotation around the origin by a uniform random angle in $[0,2\pi]$ and a translation by a uniform vector in $[0,1]^2$.
Consider the URM $\phi(\Z^2)$ and let $\calH$ be $\{\phi(\ell): \ell \text{ is a horizontal line in } \Z^2\}$. An element of $\calH$ is a line in $\R^2$, but it naturally corresponds to a biinfinite path in $\phi(\Z^2)$. The $\phi$-image of the positive vertical half-axis in $\Z^2$ determines an ``upward'' direction in $\phi(\Z^2)$, also applicable to the plane. We can talk about elements of $\calH$ being above or below each other, and about the {\it upper} half-plane $\HP_\ell$ of $\R^2$ determined by 
a line $\ell\in\calH$. Let $\ell^+\in\calH$ be the line that is ``above'' $\ell$ at distance 1. Let $S(\ell)$ be the {\it stripe} $\HP_\ell\setminus \HP_{\ell^+}$. Recall that for the percolation configuration $\omega\subset G$, 
$\HP_\ell(\omega)$ stands for the restriction of $\omega$ to the half-plane $\HP_\ell$, and similarly for
$\HP_\ell(\omega^\star)$. Finally, $\HP_\ell(\omega)^\infty$ and $\HP_\ell(\omega^\star)^\infty$
will stand for the union of infinite clusters in $\HP_\ell(\omega)$ and $\HP_\ell(\omega^\star)$ respectively. 

In what follows we will repeatedly use the fact that decorating a URM in $\R^2$ using a jointly isometry-invariant random process yields a unimodular random network (a decorated URG), to which the MTP can be applied.

Suppose by contradiction that for some $\ell\in\calH$ both $\HP_\ell(\omega)$ and $\HP_\ell(\omega^\star)$ have some infinite clusters with positive probability. 

\smallskip
{\bf Case 1: $S(\ell)\cap V(\HP_\ell(\omega)^\infty)$ is finite.} This has zero probability by the MTP applied to $\Z^2$. Namely, for each such $\ell$ and vertex $x$ of $\Z^2$ with $\phi (x)\in\ell$, assign to $x$ a uniformly chosen $y\in\Z^2$ such that $\phi (y)\in\ell$ and the distance of $\phi(y)$ from $S(\ell)\cap \HP_\ell(\omega)^\infty$ is at most 2. 

\smallskip
{\bf Case 2: $S(\ell)\cap V(\HP_\ell(\omega)^\infty)$ is infinite, and an infinite cluster $C$ of $\HP_\ell(\omega)^\infty$ has $V(C)\cap S(\ell)$ infinite.} 
Any half-stripe of $S(\ell)$ has to contain some point from $V(C)$, otherwise there would be a point $x$ in $V(C)\cap S(\ell)$ that is extremal to some line perpendicular to $\ell$ (meaning that $x$ is on the line and all other elements of $V(C)\cap S(\ell)$ are on one side of the line or on the line), and we could assign to each vertex of $V(C)\cap S(\ell)$ one of these points, chosen at random from the finitely many possibilities. This would contradict the MTP, applied to $G$.
So there are vertices from $V(C)\cap S(\ell)$ arbitrarily far in both directions in the stripe $S(\ell)$. There is a path in $C(\subset \HP_\ell)$ between any two of these points, and the union $U$ of these paths either separates an infinite cluster $\hat C$ of $\HP_\ell(\omega^\star)^\infty$ from $\ell$ ({\it Subcase 1}), or $\hat C$ is in between $U$ and $\ell$, that is, 
$\hat C$ separates $C$ from $\ell$ ({\it Subcase 2}). See Figure \ref{abra}. In either subcase, the cluster that is separated cannot intersect $\ell$. Consider Subcase 1, when $\hat C$ is separated from $\ell$. In particular, $\hat C$ is also a cluster of $\omega^\star$.  So there is an infinite cluster in $\omega^\star$ that intersects some stripe $S(\ell')$, 
$\ell'\in\calH$, but does not intersect any stripe $S(\ell'')$ with $\ell''$ below $\ell'$. Refer to this by saying that $S(\ell')$ is the lowermost stripe (for this cluster).
if (B) holds, by changing the $\omega$-status of a suitable finite set of edges in and incident to a path that connects this lowermost stripe to some stripe below it, we get finitely many lowest lying vertices in the cluster, giving a MTP contradiction. So when (B) holds, we can rule out Subcase 1. Assume now that (A) holds. Suppose first that with positive probability there are finitely many lowermost stripes $S_1,\ldots, S_k$ (for respective infinite clusters). Then to every vertex on a vertical line $m$ of $\Z^2$ assign a uniformly chosen vertex of $\Z^2$ in $m\cap\phi^{-1}(\cup S_i))$.
This contradicts the MPT for $\Z^2$. Finally, suppose that with positive probability there are infinitely many lowermost stripes. Then $\omega^\star$ necessarily has infinitely many infinite clusters. They are all 2-ended, hence from
Lemma \ref{2ended} we get a contradiction to (A). 
For Subcase 2
we get a contradiction similarly to Subcase 1, but the lemma is not needed to contradict (A) in the primal graph.

\begin{figure} 
\begin{minipage}{\columnwidth}
\centering
{%
\begin{overpic}[scale=0.1]{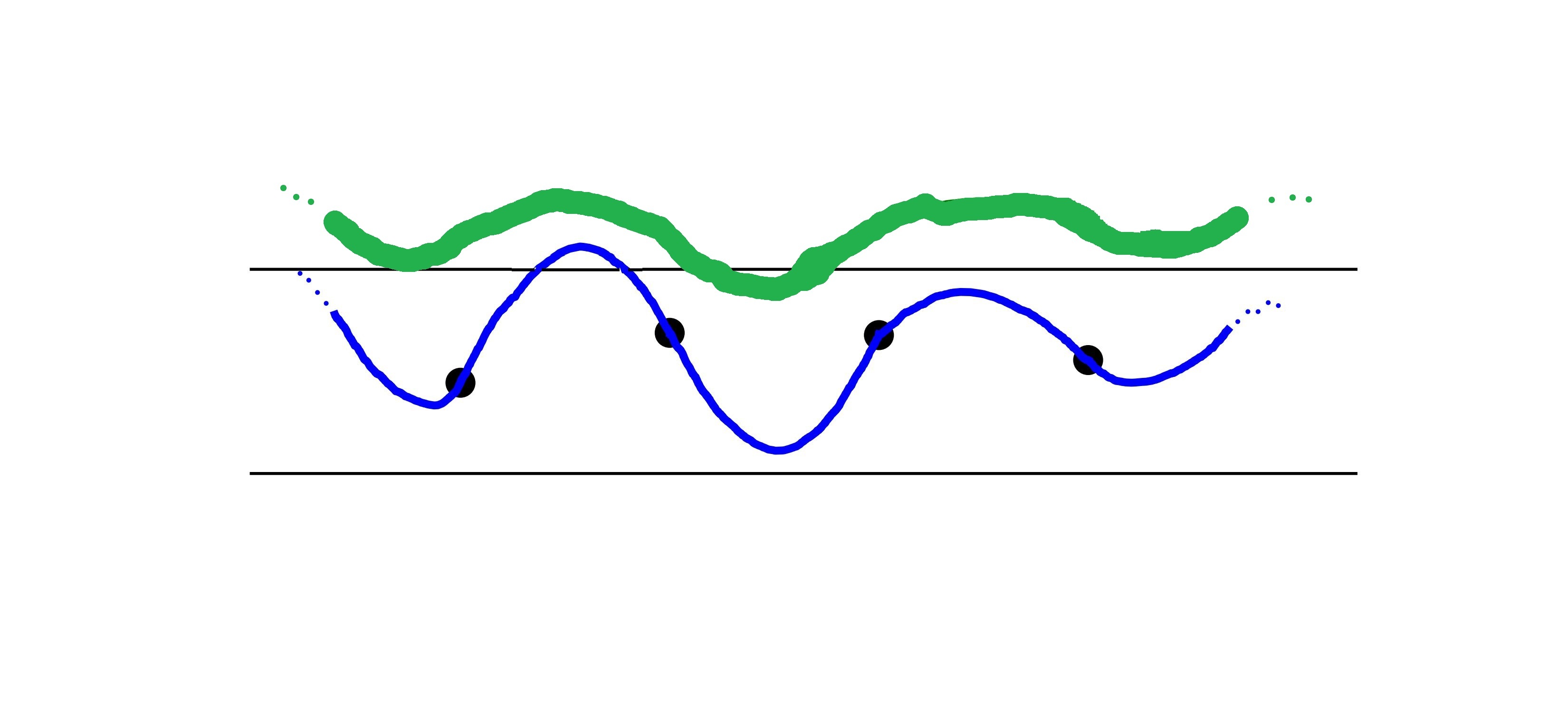}
\put(17,10){$\ell$}
\put(17,23){$\ell$}
\put(81,20){\color{myblue}$C$\color{black}}
\put(81,34){\color{mygreen}$\hat C$\color{black}}
\end{overpic}
\hspace{-15mm}
\begin{overpic}[scale=0.1]{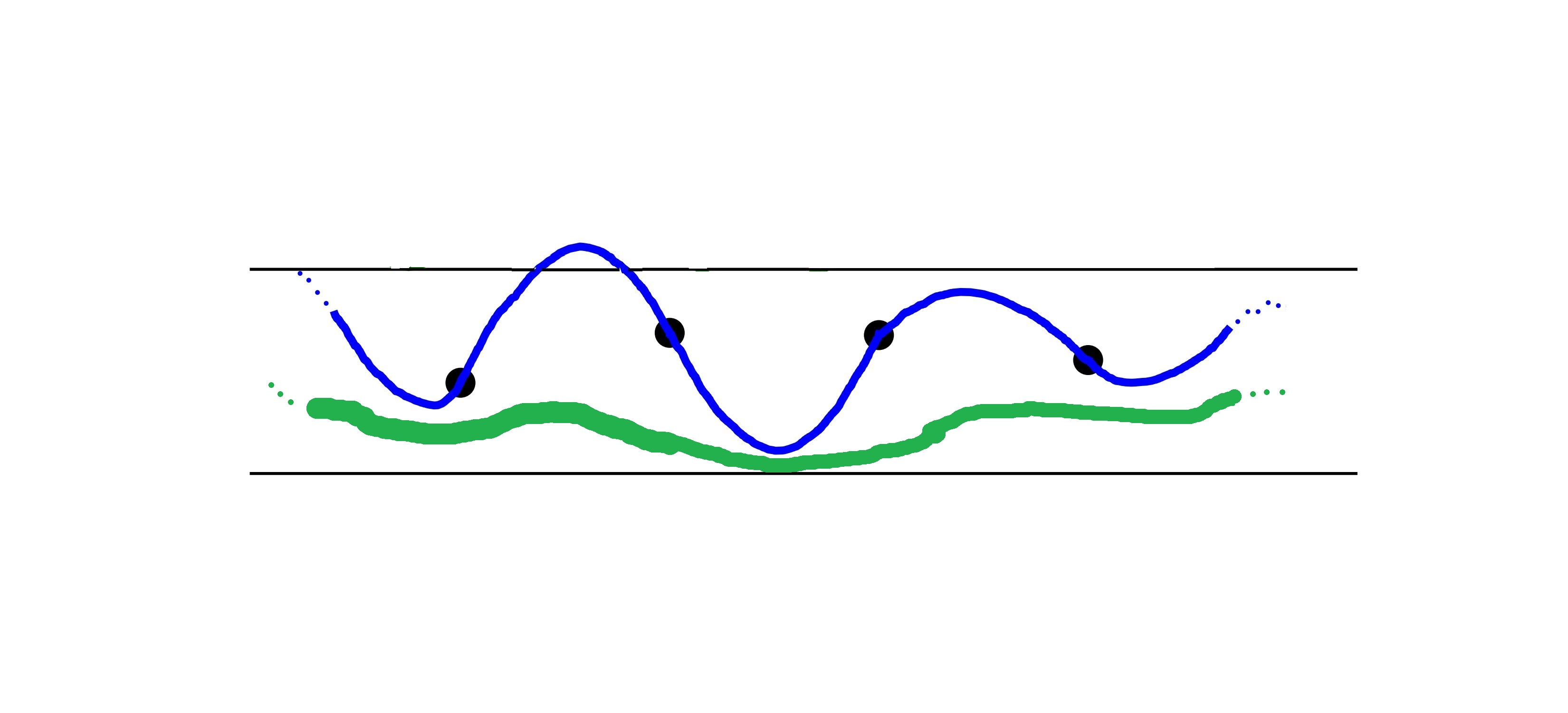}
\put(14,10){$\ell$}
\put(14,23){$\ell^+$}
\put(83,24){\color{myblue}$C$\color{black}}
\put(83,17){\color{mygreen}$\hat C$\color{black}}
\end{overpic}
}
\end{minipage}
\caption{Subcases 1 and 2.}
\label{abra}
\end{figure}
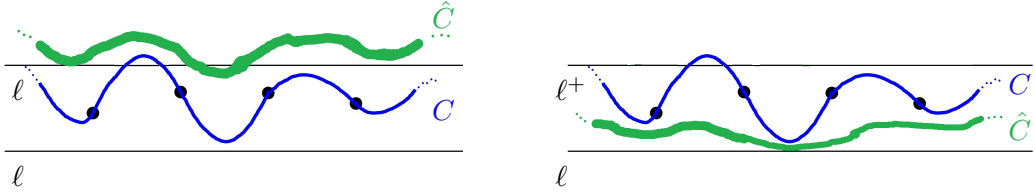

\smallskip
{\bf Case 3: $S(\ell)\cap V(\HP_\ell(\omega)^\infty)$ is infinite, and for every infinite cluster $C$ of $\HP_\ell(\omega)^\infty$, $V(C)\cap S(\ell)$ is finite.} In particular, $\HP_\ell(\omega)^\infty$ has infinitely many infinite clusters. Let us first assume (A). As $\omega$ has only finitely many infinite clusters, infinitely many infinite clusters of $\HP_\ell(\omega)^\infty$ have to belong to the same $\omega$-cluster $K$. In particular, there are distinct clusters $C_1,C_2$ in $\HP_\ell(\omega)^\infty$, a cluster $\hat C$ in $\HP_\ell(\omega^\star)^\infty$ that separates them, and
a biinfinite path in $\omega$ that has finitely many vertices outside of $\HP_\ell$ and one ray in $C_1$, the other one in $C_2$. Then $\hat C$ is in an $\omega^\star$-cluster that has finitely many lowermost vertices. This contradicts the MTP applied to $G^\star$. Now assume (B), and apply the same argument, after inserting a path between distinct clusters $C_1,C_2$ of $\HP_\ell(\omega)^\infty$ if necessary, to make them be part of the same cluster of $\omega$.


\medskip
{\bf Suppose now that $G$ is a hyperbolic URM in the hyperbolic plane.} Then half-plane coexistence follows from Theorem \ref{hyper}.
\end{proofof}\qed

For Theorem \ref{hyper}, we will need a few lemmas. The next one is a consequence of the continuous version of the Mass Transport Principle, Theorem 5.2 in \cite{BS}.

\begin{lemma}\label{teljes}
Let $K\subset \HH^2$ be a random closed convex set whose law is invariant under the isometries of $\HH^2$. Then $K\in\{\emptyset,\HH^2\}$ almost surely.
\end{lemma}

\begin{proof} Assume with positive probability that $\emptyset \neq K \neq \mathbb H^2$. Fix $o\in\HH^2$. Since $K$ is closed and convex, the distance from $o$ to $\partial K$ is attained at a unique point of $K$ almost surely.
Define an isometry-invariant mass transport by letting each point of $\HH^2$ send one unit of mass to the point of $K$ closest to it.
The point $o$ sends exactly one unit of mass. However, if $o\in\partial K$, it receives mass from all points whose closest boundary point is $o$, a set of infinite hyperbolic area. Thus the expected mass received at $o$ is infinite, contradicting the continuous Mass Transport Principle. Hence the assumption was false.
\end{proof}\qed

\begin{lemma}\label{forest}
Suppose that $\omega$ has some cluster with infinitely many ends. Then there is an invariant subforest $F$ in $\omega$ that has some cluster with infinitely many ends, and for any three ends of such a cluster there is a biinfinite path $P$ representing two of these ends such that both components of $\HH^2\setminus P$ contain an infinite cluster from $\omega^\star$.
\end{lemma}

\begin{proof}
Call a ball $B(x,r)$ of radius $r$ around $x$ in the graph $\omega$ a {\it furcation ball} if its removal from $\omega$ results in at least 3 new infinite clusters. Whenever $B$ is a furcation ball and $C_1$ and $C_2$ are two new infinite clusters in $\omega\setminus B$, there are two infinite clusters $S_1(C_1,C_2)$ and $S_2(C_1,C_2)$ in $\omega^\star$ that separate $C_1$ and $C_2$ in $\HH^2\setminus B$ (because of Lemma \ref{2ended}, there is an infinite cluster in $\omega^\star\cup\{e^\star: e \text{ has an endpoint in } B
\}$ that separates $C_1$ and $C_2$).
If $r$ is large enough, there exists a furcation ball, and by the MTP there are infinitely many. Fix $r$ and select a nonempty invariant set $\calB$ of pairwise disjoint furcation balls. Let $F_0$ be the wired uniform spanning forest of $\omega\setminus\cup_{B\in \calB} B$. Every cluster of $F_0$ has at most 2 ends \cite{AL}. Take a random uniform spanning tree in each of the balls $B\in\calB$, and let $F$ be the union of these and $F_0$. There is some component of infinitely many ends in $F$, because every furcation ball splits its component to at least 3 infinite parts. Since $F_0$ had clusters of at most 2 ends, for any three disjoint infinite paths $P_1,P_2,P_3$ in $F$ there are distinct $i,j\in\{1,2,3\}$ such that $P_i$ and $P_j$ are in different components of $\omega\setminus\cup_{B\in \calB} B$, and hence in different components of $\omega\setminus B$ for some $B\in \calB$. Calling these components $C_1$ and $C_2$, consider any biinfinite path $P$ with all but finitely many vertices in $C_1\cup C_2$. Then one component of $\HH^2\setminus P$ contains $S_1(C_1,C_2)$, and the other one contains $S_2(C_1,C_2)$.
\end{proof}\qed

\begin{lemma}\label{ends}
Let $\omega$ be a random subgraph of $\HH^2$ given by an invariant percolation. Suppose 
that $\omega$ has some cluster with infinitely many ends. Then the restriction $\HP (\omega)$ of $\omega$ to a fixed halfplane $\HP$ has some infinite cluster, and the restriction of its dual $\omega^\star$ to $\HP$ also contains some infinite cluster almost surely.
\end{lemma}

\begin{proof} 
Let $F$ be the subforest of $\omega$ given by Lemma \ref{forest}. It is enough to prove that $\HP(F)$ has at least 3 ends, because then by Lemma \ref{forest} there is a biinfinite path $P$ in $\HP(F)\subset \HP(\omega)$ such that the component of $\HH^2\setminus P$ inside $\HP$ contains an infinite cluster from $\omega^\star$.

Suppose now that with probability $p>0$ all component trees of $\HP (F)$ have only 0, 1 or 2 ends. Consider a subcritical Poisson line percolation in the hyperbolic plane $\HH$, see \cite{P, BGR} for background.  All we are using is that this is an ergodic isometry-invariant random set $\ELL$ of countably many lines such that 
every bounded line segment in $\HH$ is intersected by only finitely many elements of $\ELL$. For every $L\in\ELL$ take an iid uniform $[0,1]$ random variable $\xi_L$, and let $\ELL _c:=\{L\in\ELL, \xi_L\leq c\}$ for every $c\in[0,1]$. Each $L\in\ELL$ divides $\HH$ into two open half-planes $\HP_L^1,\HP_L^2$.
Now, all the $\HP_L^i$ ($i\in\{1,2\}$) are isometric to the fixed half-plane $\HP$, hence by isometry-invariance the distribution of $F$ restricted to any of the $\HP_L^i$ is the same. 
Call $\HP_L^i$ finitary when all the clusters of $\HP_L^i(F)$ have 0, 1 or 2 ends.
Consider the union of all the finitary half-planes. This is an invariant random subset of $\HH^2$ that is nonempty with probability at least $p$. Apply Lemma \ref{teljes} to its complement, to obtain that the union of finitary half-planes is $\HP$ with probability at least $p$, and hence (by ergodicity) almost surely. Therefore the restriction of $F$ to any connected component of $\HH\setminus \cup_{L\in\ELL_c} L$ has only clusters with $\leq 2$ ends almost surely, because it is the intersection of the forests $\HH_L^i (F)$ ($L\in\ELL_c$, $i\in\{1,2\}$) inside the cluster, and they all have at most 2 ends. Taking a monotone sequence $c_n\to 0$, we get an exhaustion of $F$ by {\it invariant} subforests of clusters that have at most 2 ends. This is not possible, because $F$ is has infinitely many ends. (The expected degrees in the backbone of $F$ are strictly greater than 2, while the expected degrees in the restrictions in the exhaustions are all less than or equal to 2. Here the backbone refers to the subgraph induced by edges of $F$ that are contained in a biinfinite path of $F$.)
\end{proof}\qed

\begin{proofof}{Theorem \ref{hyper}}
By assumption both $\omega$ and $\omega^\star$ have some infinite clusters. Thus either $\omega$ or $\omega^\star$ must have infinitely many infinite clusters (see Corollary 3.6 in `cite{BS}, which generalizes to URM's). By symmetry we may assume it is $\omega$. Then a large enough ball in $\HH^2$ intersects at least 3 infinite components in $\omega$, showing that there exist furcation balls in the dual. Hence $\omega^\star$ has a cluster with infinitely many ends.
Apply Lemma \ref{ends}.
\end{proofof}\qed

\bigskip

\noindent
{\bf Acknowledgments:} I am grateful to G\'abor Pete for stimulating conversations. I also thank P\'eter Mester for useful comments. This research was partially supported by Icelandic  Research Fund grant 239736-051 and the ERC
grant No. 810115-DYNASNET.

\ \\

\noindent
{\bf \'Ad\'am Tim\'ar}\\
Division of Mathematics, The Science Institute, University of Iceland\\
Dunhaga 3 IS-107 Reykjavik, Iceland\\
and\\
Alfr\'ed R\'enyi Institute of Mathematics\\
Re\'altanoda u. 13-15, Budapest 1053 Hungary\\
\texttt{madaramit[at]gmail.com}\\

\end{document}